\documentclass[12pt, a4paper, bibliography=totocnumbered, abstracton]{scrartcl} 

\usepackage{amsmath,amssymb,amsthm,mathtools,stmaryrd} 
\usepackage{geometry}
\usepackage{hyperref} 
\usepackage{cite}

\usepackage{tikz}
\usetikzlibrary{decorations.pathreplacing}
\usepackage[utf8]{inputenc}
\usepackage{lmodern}

\newtheorem{theorem}{Theorem}[section]

\newtheorem{lemma}[theorem]{Lemma}
\newtheorem{conjecture}[theorem]{Conjecture}

\DeclareMathOperator{\ex}{ex}

\usepackage{bbm}
\usepackage{autonum}

\numberwithin{equation1}{section}
\begin{document}
\title{Making $K_{r+1}$-Free Graphs $r$-partite}

\author{J\'ozsef Balogh, Felix Christian Clemen, Mikhail Lavrov, Bernard Lidick\'{y}, Florian Pfender}

\author{%
  J\'ozsef Balogh \footnote{Department of Mathematics, University of Illinois at Urbana-Champaign, Urbana, Illinois 61801, USA, and Moscow Institute of Physics and Technology, Russian Federation. E-mail: \texttt{jobal@illinois.edu}. Research is partially supported by NSF Grant DMS-1764123, Arnold O. Beckman Research
Award (UIUC Campus Research Board RB 18132) and the Langan Scholar Fund (UIUC).}%
\and Felix Christian Clemen \footnote {Department of Mathematics, University of Illinois at Urbana-Champaign, Urbana, Illinois 61801, USA, E-mail: \texttt{fclemen2@illinois.edu}.}%
  \and Mikhail Lavrov \footnote {Department of Mathematics, University of Illinois at Urbana-Champaign, Urbana, Illinois 61801, USA, E-mail: \texttt{mlavrov@illinois.edu}}
  \and Bernard Lidick\'{y} \footnote {Iowa State University, Department of Mathematics, Iowa State University, Ames, IA.,E-mail: \texttt{lidicky@iastate.edu}. Research of this author is partially supported by NSF grant
DMS-1855653.}
\and Florian Pfender \footnote{Department of Mathematical and Statistical Sciences, University of Colorado Denver, E-mail: \texttt{Florian.Pfender@ucdenver.edu}. Research of this author is partially supported by NSF grant DMS-1855622.}%
  }

\date{}

\maketitle
\abstract{
The Erd\H{o}s--Simonovits stability theorem states that for all $\varepsilon>0$ there exists $\alpha>0$ such that if $G$ is a $K_{r+1}$-free graph on $n$ vertices with $e(G) > \ex(n,K_{r+1}) - \alpha n^2$, then one can remove $\varepsilon n^2$ edges from $G$ to obtain an $r$-partite graph. F{\"u}redi gave a short proof that one can choose $\alpha=\varepsilon$. We give a bound for the relationship of $\alpha$ and $\varepsilon$ which is asymptotically sharp as $\varepsilon \to 0$.
}

\section{Introduction}
Erd\H{o}s asked how many edges need to be removed in a triangle-free graph on $n$ vertices in order to make it bipartite. He conjectured that the balanced blow-up of $C_5$ with class sizes $n/5$ is the worst case, and hence $1/25 n^2$ edges would always be sufficient. Together with Faudree, Pach and Spencer \cite{Howtotriangle}, he proved that one can remove at most $1/18n^2$ edges to make a triangle-free graph bipartite. \\ 
Further, Erd\H{o}s, Gy\H{o}ri and Simonovits \cite{Howmany} proved that for graphs with at least $n^2/5$ edges, an unbalanced $C_5$ blow-up is the worst case. For $r\in \mathbb{N}$, denote $D_r(G)$ the minimum number of edges which need to be removed to make $G$ $r$-partite.

\begin{theorem} [Erd\H{o}s, Gy\H{o}ri and Simonovits \cite{Howmany}]
\label{ErdosC_5blowup}
Let $G$ be a $K_3$-free graph on $n$ vertices with at least $n^2/5$ edges. There exists an unbalanced $C_5$ blow-up of $H$ with $e(H)\geq e(G)$ such that   
\begin{align}
D_2(G)\leq D_2(H).
\end{align} 
\end{theorem}

This proved the Erd\H{o}s conjecture for graphs with at least $n^2/5$ edges. A simple probabilistic argument (e.g.~\cite{Howmany}) settles the conjecture for graphs with at most $2/25 n^2$ edges.\\
A related question was studied by Sudakov; he determined the maximum number of edges in a $K_4$-free graph which need to be removed in order to make it bipartite \cite{K4freebip}.
This problem for $K_6$-free graphs was solved by Hu, Lidick\'{y}, Martins, Norin and Volec~\cite{hulimanovo}.
\\
We will study the question of how many edges in a $K_{r+1}$-free graph need at most to be removed to make it $r$-partite. For $n\in \mathbb{N}$ and a graph $H$, $\ex(n,H)$ denote the Tur\'an number, i.e.~the maximum number of edges of an $H$-free graph. 
The Erd\H{o}s--Simonovits theorem \cite{Erdossimonovits} for cliques states that for every $\varepsilon>0$ there exists $\alpha >0$ such that if $G$ is a $K_{r+1}$-free graph on $n$ vertices with $e(G) > \ex(n,K_{r+1}) - \alpha n^2$, then $D_r(G)\leq \varepsilon n^2$.

F{\"u}redi \cite{Furedi} gave a nice short proof of the statement that a $K_{r+1}$-free graph $G$ on $n$ vertices with at least $\ex(n,K_{r+1})-t$ edges satisfies $D_r(G)\leq t$; thus providing a quantitative version of the Erd\H{o}s--Simonovits theorem. 
In \cite{hulimanovo} F{\"u}redi's result was strengthened for some values of $r$.
For small $t$, we will determine asymptotically how many edges are needed. For very small $t$, it is already known \cite{smallalpha} that $G$ has to be $r$-partite.  

\begin{theorem}[Brouwer \cite{smallalpha}]
\label{smalllalpha}
Let $r\geq 2$ and $n\geq 2r+1$ be integers. Let $G$ be a $K_{r+1}$-free graph on $n$ vertices with  $e(G)\geq \ex(n,K_{r+1})-\left\lfloor \frac{n}{r} \right\rfloor+2$. Then
\begin{align}
    D_r(G)=0.  
\end{align}
 \end{theorem}
This result was rediscovered in \cite{smallalpha3,smallalpha4,smallalpha5,smallalpha2}. We will study $K_{r+1}$-free graphs on fewer edges. 

\begin{theorem}
\label{trianglefreemain}
Let $r\geq 2$ be an integer. Then for all $n \geq 3r^2$ and for all $0\leq \alpha \leq 10^{-7}r^{-12}$ the following holds. Let $G$ be a $K_{r+1}$-free graph on $n$ vertices with 
\begin{align}
\label{cliquecondition}
    e(G)\geq \ex(n,K_{r+1})-t,
\end{align}    
    where $t=\alpha n^2$, then  
\begin{align}
D_r(G)\leq \left(\frac{2r}{3\sqrt{3}} + 30r^3\alpha^{1/6} \right)\alpha^{3/2} n^2. 
\end{align}
 
\end{theorem}
Note that we did not try to optimize our bounds on $n$ and $\alpha$ in the theorem. One could hope for a slightly better error term of $30r^3\alpha^{5/3}$ in Theorem~\ref{trianglefreemain}, but the next natural step would be to prove a structural version. \\
To state this structural version we introduce some definitions. The blow-up of a graph $G$ is obtained by replacing every vertex $v\in V(G)$ with finitely many copies so that the copies of two vertices are adjacent if and only if the originals are. \\
For two graphs $G$ and $H$, we define $G \otimes H$ to be the graph on the vertex set $V(G)\cup V(H)$ with $gg'\in E(G \otimes H)$ iff $gg'\in E(G)$, $hh'\in E(G \otimes H)$ iff $hh'\in E(H)$ and $gh\in E(G \otimes H)$ for all $g\in V(G)$, $h \in V(H)$. 

\begin{conjecture}
\label{K_rstructuralconj}
Let $r\geq 2$ be an integer and $n$ sufficiently large. Then there exists $\alpha_0>0$ such that for all $0\leq \alpha \leq \alpha_0$ the following holds. For every $K_{r+1}$-free graph $G$ on $n$ vertices there exists an unbalanced $K_{r-2} \otimes C_5$ blow-up $H$ on $n$ vertices with $e(H)\geq e(G)$ such that  
\begin{align}
D_r(G)\leq D_r(H).
\end{align} 
\end{conjecture}
This conjecture can be seen as a generalization of Theorem~\ref{ErdosC_5blowup}. We will prove that Theorem~\ref{trianglefreemain} is asymptotically sharp by describing an unbalanced blow-up of $K_{r-2} \otimes C_5$ that needs at least that many edges to be removed to make it $r$-partite. This gives us a strong evidence that Conjecture~\ref{K_rstructuralconj} is true.

\begin{theorem}
\label{K_rexample}
Let $r,n\in \mathbb{N}$ and $0<\alpha < \frac1{4r^4}$. Then there exists a $K_{r+1}$-free graph on $n$ vertices with 

$$e(G) \geq \ex(n,K_{r+1})  -\alpha n^2 + \frac{4r}{3\sqrt{3}} \alpha^{3/2}n^2  -\frac{2r(r-3)}{9}\alpha^2 n^2$$ 
\noindent
and $$D_r(G)\geq \frac{2r}{3\sqrt{3}}  \alpha^{3/2}n^2.$$
\end{theorem}
In Kang-Pikhurko's proof \cite{smallalpha5} of Theorem~\ref{smalllalpha} the case $e(G)=\ex(n,K_{r+1})-\left\lfloor n/r \right\rfloor+1$ is studied. In this case they constructed a family of $K_{r+1}$-free non-$r$-partite graphs, which includes our extremal graph, for that number of edges.\\
We recommend the interested reader to read the excellent survey \cite{surveyNiki} by Nikiforov. He gives a good overview on further related stability results, for example on guaranteeing large induced $r$-partite subgraphs of $K_{r+1}$-free graphs. \\
We organize the paper as follows. In Section~\ref{sectionK_rfree} we prove Theorem~\ref{trianglefreemain} and in Section~\ref{sectionK_rexample} we give the sharpness example, i.e. we prove Theorem~\ref{K_rexample}.

\section{Proof of Theorem~\ref{trianglefreemain} }
\label{sectionK_rfree}
Let $G$ be an $n$-vertex $K_{r+1}$-free graph with $e(G)\geq \text{ex}(n,K_{r+1})-t$, where $t=\alpha n^2$. We will assume that $n$ is sufficiently large. Furthermore, by Theorem~\ref{smalllalpha} we can assume that 
\begin{align}
    \label{smallalpha}
    \alpha \geq \frac{ \left \lfloor \frac{n}{r} \right \rfloor-2}{n^2} \geq \frac{1}{2rn}.
\end{align}
This also implies that $t \geq r$ because $n\geq 3r^2$. During our proof we will make use of Tur\'an's theorem and a version of Tur\'an's theorem for $r$-partite graphs multiple time. Tur\'an's theorem  \cite{Turanstheorem} determines the maximum number of edges in a $K_{r+1}$-free graph.

\begin{theorem}[Tur\'an~\cite{Turanstheorem}]
Let $r\geq 2$ and $n\in \mathbb{N}$. Then,
$$\frac{n^2}{2} \left(1-\frac{1}{r} \right) - \frac{r}{2} \leq \text{ex}(n,K_{r+1}) \leq \frac{n^2}{2} \left(1-\frac{1}{r}\right).$$
\end{theorem}

Denote $K(n_1,\ldots,n_r)$ the complete $r$-partite graph whose $r$ color classes have sizes
$n_1,\ldots,n_r$, respectively. Turan\'s theorem for $r$-partite graphs states the following.

\begin{theorem}[folklore]
\label{Turanrpartite}
Let $r\geq 2$ and $n_1,\ldots,n_r \in \mathbb{N}$ satisfying $n_1\leq \ldots \leq n_r$. For a $K_r$-free subgraph $H$ of $K(n_1,\ldots,n_r)$, we have 

$$ e(H)\leq e(K(n_1,...,n_r))-n_1n_2.$$
\end{theorem}
For a proof of this folklore result see for example \cite[Lemma 3.3]{sparseKr+1}.

We denote the maximum degree of $G$ by $\Delta(G)$.
For two disjoint subsets $U,W$ of $V(G)$, write $e(U,W)$ for the number of edges in $G$ with one endpoint in $U$ and the other endpoint in $W$. We write $e^c(U,W)$ for the number of non-edges between $U$ and $W$, i.e. $e^c(U,W)=|U||W|-e(U,W)$.\\
F{\"u}redi~\cite{Furedi} used Erd\H{o}s' degree majorization algorithm \cite{Erdosdegreealg} to find a vertex partition with some useful properties. We include a proof for completeness. 

\begin{lemma}[F{\"u}redi~\cite{Furedi}]
Let $t,r,n\in \mathbb{N}$ and $G$ be an $n$-vertex $K_{r+1}$-free graph with $e(G)\geq \ex(n,K_{r+1})-t$. Then there exists a vertex partition $V(G)=V_1\cup \ldots \cup V_r$ such that
\begin{align}
    \sum_{i=1}^re(G[V_i])\leq t, \quad \Delta(G)=\sum_{i=2}^r |V_i| \quad \text{and} \quad  \sum_{1\leq i < j \leq r}e^c(V_i,V_j) \leq 2t. \label{eq:nonedgs} 
\end{align}
\end{lemma}
\begin{proof}
Let $x_1\in V(G)$ be a vertex of maximum degree. Define $V_1:=V(G)\setminus N(x_1)$ and $V_1^+=N(x_1)$. Iteratively, let $x_i$ be a vertex of maximum degree in $G[V_{i-1}^+]$. Let $V_i:=V_{i-1}^+ \setminus N(x_i)$ and $V_i^+=V_{i-1}^+ \cap N(x_i)$. Since $G$ is $K_{r+1}$-free this process stops at $i\leq r$ and thus gives a vertex partition $V(G)=V_1 \cup \ldots \cup V_r$. \\
In the proof of \cite[Theorem 2]{Furedi}, it is shown that the partition obtained from this algorithm satisfies \[ \sum_{i=1}^r e(G[V_i]) \le t. \] By construction, \[ \sum_{i=2}^r |V_i| =  |V_1^+| = |N(x_1)| = \Delta(G). \] 
Let $H$ be the complete $r$-partite graph with vertex set $V(G)$ and all edges between $V_i$ and $V_j$ for $1\leq i< j \leq r$. The graph $H$ is $r$-partite and thus has at most $\ex(n, K_{r+1})$ edges. Finally, since $G$ has at most $t$ edges not in $H$ and at least $\ex(n, K_{r+1}) - t$ edges total, at most $2t$ edges of $H$ can be missing from $G$, giving us \[ \sum_{1 \le i < j \le r} e^c(V_i, V_j) \le 2t \] and proving the last inequality.
\end{proof}
For this vertex partition we can get bounds on the class sizes.

\begin{lemma}
For all $i\in [r]$, $|V_i|\in \{\frac{n}{r}-\frac{5}{2}\sqrt{\alpha} n, \frac{n}{r}+\frac{5}{2}\sqrt{\alpha}n  \}$ and thus also $$\Delta(G)\leq \frac{r-1}{r}n+ \frac{5}{2}\sqrt{\alpha}n.$$
\end{lemma}
\begin{proof}
We know that 
\[
	\sum_{1\leq i < j \leq r} |V_i||V_j| \geq e(G)- \sum_{i=1}^r e(G[V_i]) \geq \left(1-\frac{1}{r}  \right)  \frac{n^2}{2}-\frac{r}{2}-2t.
\]
Also, 
\[
	\sum_{1\leq i < j \leq r} |V_i||V_j|= \frac{1}{2} \sum_{i=1}^r |V_i|(n-|V_i|) = \frac{n^2}{2}-\frac{1}{2} \sum_{i=1}^r |V_i|^2.
\]
Thus, we can conclude that
\begin{equation}
 \label{eq:classsum} 
\sum_{i=1}^r |V_i|^2 \leq \frac{n^2}{r}+r+4t.
\end{equation}
Now, let $x=|V_1|-n/r$. Then,
\begin{align}
    \sum_{i=1}^r |V_i|^2 &=  \left(\frac{n}{r}+x\right)^2 + \sum_{i=2}^r |V_i|^2  \geq \left( \frac{n}{r}+ x \right)^2 + \frac{\left(  \sum_{i=2}^r |V_i|\right)^2}{r-1} \\
    &\geq \left( \frac{n}{r}+x \right)^2 + \frac{\left(n \left(1-\frac{1}{r} \right)-x \right)^2}{r-1} 
\geq \frac{n^2}{r} + x^2.
\end{align} 
Combining this with \eqref{eq:classsum}, we get $|x| \le \sqrt{r + 4t} \le \frac52 \sqrt t = \frac52 \sqrt{\alpha}n$, and thus
\[
	\frac{n}{r}-\frac{5}{2}\sqrt{\alpha} n \leq |V_1|\leq \frac{n}{r}+\frac{5}{2}\sqrt{\alpha} n.
\]
In a similar way we get the bounds on the sizes of the other classes. 
\end{proof}

\begin{lemma}
\label{goodKr}
The graph $G$ contains $r$ vertices $x_1\in V_1 ,\ldots,x_r\in V_r$ which form a $K_r$ and for every $i$
$$\deg(x_i)\geq n-|V_i|- 5r\alpha n.  $$ 
\end{lemma}
\begin{proof}
Let $V_i^c:= V(G)\setminus V_i$. We call a vertex $v_i\in V_i$ \textit{small} if $|N(v_i) \cap V_i^c|< |V_i^c|-5r\alpha n$ and \textit{big} otherwise. For $1\leq i \leq r$, denote $B_i$ the set of big vertices inside class $V_i$. There are at most 
$$ \frac{4t}{5r\alpha n} = \frac{4}{5r}n$$
small vertices in total as otherwise \eqref{eq:nonedgs} is violated. Thus, in each class there are at least $n/10r$ big vertices, i.e. $|B_i|\geq n/10r$. The number of missing edges between the sets $B_1,\dots,B_r$ is at most $2t<\frac{1}{100r^2}n^2$. Thus, using Theorem~\ref{Turanrpartite} we can find a $K_r$ with one vertex from each $B_i$. 
\end{proof}

\begin{lemma}
\label{partition2}
There exists a vertex partition $V(G)=X_1\cup \ldots \cup X_r \cup X$ such that all $X_i$s are independent sets, $|X|\leq 5r^2 \alpha n$ and $$\frac{n}{r} - 3 \sqrt{\alpha}n \leq |X_i|\leq \frac{n}{r} +3r\sqrt{\alpha}n $$
for all $1 \leq i \leq r$.
\end{lemma}
\begin{proof}
By Lemma~\ref{goodKr} we can find vertices $x_1,\dots, x_r$ forming a $K_r$ and having $\deg(x_i)\geq n-|V_i|-5r\alpha n.$ Define $X_i$ to be the common neighborhood of $x_1,\dots, x_{i-1}$ $,x_{i+1},\dots,x_r$ and $X=V(G)\setminus (X_1 \cup \cdots \cup X_r)$. Since $G$ is $K_{r+1}$-free, the $X_i$s are independent sets. Now we bound the size of $X_i$ using the bounds on the $V_i$s. 
Since every $x_j$ has at most $|V_j|+5r\alpha n$ non-neighbors, we get
\begin{align} 
\label{degsum1} |X_i|\geq n - \sum_{\substack{1 \leq j \leq r \\ j\neq i}} \left( |V_j|+5r\alpha n \right) \geq |V_i| - 5r^2\alpha n \geq \frac{n}{r}-3\sqrt{\alpha}n.
\end{align}
and
\begin{align} 
\label{degsum2} \sum_{i=1}^r \deg(x_i) \geq n (r-1) - 5r^2\alpha n. 
\end{align}

A vertex $v\in V(G)$ cannot be incident to all of the vertices $x_1,\dots, x_r$, because $G$ is $K_{r+1}$-free. Further, every vertex from $X$ is not incident to at least two of the vertices $x_1,\dots, x_r$. Thus,

\begin{align}
    \label{degsum3}
    \sum_{i=1}^r \deg(x_i) \leq n (r-1) - |X|.
\end{align}

\noindent
Combining \eqref{degsum2} with \eqref{degsum3}, we conclude that
$$ |X|\leq 5r^2 \alpha n.$$
\noindent
For the upper bound on the sizes of the sets $X_i$ we get

\begin{align} 
 |X_i|\leq n - \sum_{\substack{1 \leq j \leq r \\ j\neq i}} |X_j|  \leq n- \frac{r-1}{r}n +3r \sqrt{\alpha}n =  \frac{n}{r} +3r \sqrt{\alpha}n. 
\end{align}

\end{proof}

We now bound the number of non-edges between $X_1,\dots, X_r$.
\begin{lemma}
$$ \sum_{1\leq i < j \leq r}e^c(X_i,X_j) \leq t+e(X,X^c)+|X|^2 - \left(1-\frac{1}{r}\right)n|X| + r. $$
\label{nonedges1}
\end{lemma}
\begin{proof}

\begin{align}
    & \quad \frac{n^2}{2}\left(1-\frac{1}{r} \right)-\frac{r}{2}-t \leq e(G)= e(X,X^c) +e(X)+\sum_{1\leq i <j \leq r}e(X_i,X_j) \\
    &\leq e(X,X^c) + \frac{|X|^2}{2}+ \left(1-\frac{1}{r}\right) \left(\frac{(n-|X|)^2}{2}\right) - \sum_{1\leq i < j \leq r}e^c(X_i,X_j).
\end{align}
This gives the statement of the lemma. 
\end{proof}

Let $$\bar{X}=\left\{v\in X \,\middle|\, \deg_{X_1\cup \cdots \cup X_r}(v)\geq \frac{r-2}{r}n+ 3\alpha^{1/3}n    \right\} \text{\quad and \quad} \hat{X}:=X\setminus \bar{X}.$$

Let $d\in [0,1]$ such that $|\bar{X}|=d|X|$. Further, let $k\in [0,5r^2]$ such that $|X|=k \alpha n$. Now we shall give an upper bound the number of non-edges between $X_1,\dots,X_r$. 

\begin{lemma}
$$ \sum_{1\leq i < j \leq r}e^c(X_i,X_j) \leq 20r^2\alpha^{4/3}n^2    + \left(1-(1-d)\frac{1}{r}k\right)  \alpha n^2. $$
\label{nonedges2}
\end{lemma}
\begin{proof}
By Lemma~\ref{nonedges1},
 \begin{align}
     & \quad \sum_{1\leq i < j \leq r}e^c(X_i,X_j) \leq t+e(X,X^c)+|X|^2 - \left(1-\frac{1}{r}\right)n|X|+r \\
     &\leq t+d|X| \Delta(G)    + (1-d)|X|\left(\frac{r-2}{r}n+ 3\alpha^{1/3}n\right)     +|X|^2 
     - \left(1-\frac{1}{r}\right)n|X|+r \\ 
     &\leq t+d|X| \left(n \frac{r-1}{r} + \frac{5}{2}\sqrt{\alpha}n\right)    + (1-d)|X|\left(\frac{r-2}{r}n+ 3\alpha^{1/3}n\right)  \\  
     &~ +|X|^2 - \left(1-\frac{1}{r}\right)n|X|+r \\
    %   &\leq 2d|X|  \sqrt{\alpha}n    + (1-d)|X|\alpha^{0.4}n     +|X|^2 + t + n|X| \frac{ d(r-1)+(1-d)(r-2) -(r-1)}{r}  +r\\
       &\leq \frac{5}{2}d|X|  \sqrt{\alpha}n    +3 (1-d)|X|\alpha^{1/3}n     +|X|^2 + t + n|X| \frac{d-1}{r}  +r\\       
%     &\leq \frac{5}{2}d|X|  \sqrt{\alpha}n    + 3(1-d)|X|\alpha^{1/3}n     +|X|^2 + \left(1-(1-d)\frac{1}{r}k\right)  \alpha n^2 +r\\
     &\leq \frac{5}{2}k \alpha^{3/2}n^2    + 3k \alpha^{4/3}n^2     +|X|^2 + \left(1-(1-d)\frac{1}{r}k\right)  \alpha n^2 +r \\  
     &\leq \frac{25}{2}r^2\alpha^{3/2}n^2+ 15r^2\alpha^{4/3}n^2+25r^4 \alpha^2 n^2+ \left(1-(1-d)\frac{1}{r}k\right)  \alpha n^2 +r\\
     &\leq 20r^2\alpha^{4/3}n^2   + \left(1-(1-d)\frac{1}{r}k\right)  \alpha n^2.    
 \end{align} 
 %This gives us the bound we want.
 \end{proof}

\noindent
Let
\begin{align}
    C(\alpha):=20r^2\alpha^{4/3}  + \left(1-(1-d)\frac{1}{r}k\right)\alpha .
\end{align} 
For every vertex $u\in X$ there is no $K_r$ in $N_{X_1}(u) \cup \cdots \cup N_{X_r}(u)$. Thus, by applying Theorem~\ref{Turanrpartite} and Lemma~\ref{nonedges2}, we get

 %\min_i{|N_{X_i}(u)|^2} \leq%

\begin{align}
   \min_{i\neq j}{|N_{X_i}(u)| |N_{X_j}(u)|} \leq  \sum_{1\leq i < j \leq r}e^c(X_i,X_j)\leq C(\alpha)n^2.   \label{sizesineq} 
\end{align}
Bound \eqref{sizesineq} implies in particular that every vertex $u\in X$ has degree at most $\sqrt{C(\alpha)}n$ to one of the sets $X_1,\dots,X_r$, i.e.
\begin{align}
\min_i{|N_{X_i}(u)|} \leq \sqrt{C(\alpha)}n. \label{ineqpartition1} 
\end{align}
Therefore, we can partition $\hat{X}=A_1 \cup \ldots \cup A_r$ such that every vertex $u\in A_i$ has at most $\sqrt{C(\alpha)}n$ neighbors in $X_i$. \\
By the following calculation, for every vertex $u\in \bar{X}$ the second smallest neighborhood to the $X_i$'s has size at least $\alpha^{1/3}n$.
\begin{align}
\label{twosmallestsizes}
    \min_{i\neq j}{|N_{X_i}(u)| + |N_{X_j}(u)|} \geq \frac{r-2}{r}n+3 \alpha^{1/3}n -  (r-2)\left( \frac{n}{r} +3r\sqrt{\alpha}n\right) \geq 2\alpha^{1/3}n,    
\end{align}
where we used the definition of $\bar{X}$ and Lemma~\ref{partition2}. Combining the lower bound on the second smallest neighborhood with \eqref{sizesineq} we can conclude that for every $u\in \bar{X}$
\begin{align}
 \min_i{|N_{X_i}(u)|} \leq \frac{C(\alpha)}{\alpha^{1/3}}n. \label{ineqpartition2} 
\end{align}
Hence, we can partition $\bar{X}=B_1 \cup \ldots \cup B_r$ such that every vertex $u\in B_i$ has at most $C(\alpha)\alpha^{-1/3}n$ neighbors in $X_i$. Consider the partition $A_1 \cup B_1 \cup X_1$,$A_2 \cup B_2 \cup X_2, \dots, A_r \cup B_r \cup X_r$. By removing all edges inside the classes we end up with an $r$-partite graph. We have to remove at most 
  \begin{align}
     &\quad  e(X) +  d|X| \frac{C(\alpha)}{\alpha^{1/3}}n +(1-d)|X|\sqrt{C(\alpha)}n  
      \leq  6r^2\alpha^{5/3}n^2  +(1-d)k\sqrt{C(\alpha)}\alpha n^2 \\
     &\leq  6r^2\alpha^{5/3}n^2  +(1-d)k \left( \sqrt{20r^2\alpha^{4/3}}    + \sqrt{\left(1-(1-d)\frac{1}{r}k \right)  \alpha} \right) \alpha n^2 \\ 
     &\leq \left(\frac{2r}{3\sqrt{3}}+30r^3\alpha^{1/6} \right)\alpha^{3/2}n^2 
 \end{align}
  edges. We have used \eqref{ineqpartition1}, \eqref{ineqpartition2} and the fact that 
 
 $$(1-d)k \sqrt{1-(1-d)\frac{k}{r}}\leq \frac{2r}{3\sqrt{3}},$$  which can be seen by setting $z=(1-d)k$ and 
 finding the maximum of $f(z):=z \sqrt{1-\frac{z}{r}}$ which is obtained at $z=2r/3$.

\section{Sharpness Example}
\label{sectionK_rexample}
In this section we will prove Theorem~\ref{K_rexample}, i.e.~that the leading term from Theorem~\ref{trianglefreemain} is best possible. 

\begin{proof}[Proof of Theorem~\ref{K_rexample}]
Let $G$ be the graph with vertex set $V(G)=A \cup X \cup B \cup C \cup D \cup X_1 \cdots \cup X_{r-2}$, where all classes $A,X,B,C,D,X_1, \dots, X_{r-2}$ form independent sets; $A,X,B,C,D$ form a complete blow-up of a $C_5$, where the classes are named in cyclic order; and for each  $ 1\leq i\leq r-2$, every vertex from $X_i$ is incident to all vertices from $V(G)\setminus X_i$.

\begin{figure}[ht]
\begin{tikzpicture}[scale=0.3]
\node at (0,0) (f) {};
\node[draw] at (21, 10)   (a) {};
\node[draw] at (27, 10)   (b) {};
\node[draw] at (24, 20)   (c) {};
\node[draw] at (19, 15)   (d) {};
\node[draw] at (29, 15)   (e) {};

\node[draw] at (15,2) (x1) {};
\node[draw] at (20,2) (x2) {};
\node[draw] at (33,2) (x3) {};

\node[draw, circle, fill=black, scale=0.5]at (24.5,1.5) (v1) {};
\node[draw, circle, fill=black, scale=0.5] at (26,1.5) (v2) {};
\node[draw, circle, fill=black, scale=0.5] at (27.5,1.5) (v3) {};

\draw[] (x1) -- (x2); 
\draw [black]   (x1) to[out=-20,in=200] (x3);
\draw [black]   (x2) to[out=10,in=170] (x3);
\draw[] (x1) -- (a); \draw[] (x1) -- (b);  \draw[] (x1) -- (c);  \draw[] (x1) -- (d); \draw[] (x1) -- (e); 
\draw[] (x2) -- (a); \draw[] (x2) -- (b);  \draw[] (x2) -- (c);  \draw[] (x2) -- (d); \draw[] (x2) -- (e); 
\draw[] (x3) -- (a); \draw[] (x3) -- (b);  \draw[] (x3) -- (c);  \draw[] (x3) -- (d); \draw[] (x3) -- (e);

%\put(220,15){$\dots$};  
\draw [fill=white] (x1) circle[radius= 4.5em]; 
\draw [fill=white] (x2) circle[radius= 4.5em]; 
\draw [fill=white] (x3) circle[radius= 4.5em]; 

\draw[] (a) -- (b); 
\draw[] (a) -- (d); 
\draw[] (d) -- (c); 
\draw[] (c) -- (e); 
\draw[] (e) -- (b); 
\draw [fill=white] (a) circle[radius= 4.5em]; 
\draw [fill=white] (b) circle[radius= 4.5em]; 
\draw [fill=white] (c) circle[radius= 2.5em]; 
\draw [fill=white] (d) circle[radius= 3.5em]; 
\draw [fill=white] (e) circle[radius= 3.5em]; 

\node at (21, 10)   (a) {$D$};
\node at (27, 10)   (b) {$C$};
\node at (24, 20)   (c) {$X$};
\node at (19, 15)   (d) {$A$};
\node at (29, 15)   (e) {$B$};

\node at (15,2) (x1) {$X_1$};
\node at (20,2) (x2) {$X_2$};
\node at (33,2) (x3) {$X_{r-2}$};

\end{tikzpicture}
\label{fig1}
\caption{Graph $G$}
\end{figure}

The sizes of the classes are

$$ |X|=\frac{2r}{3}\alpha n, \quad  |A|=|B|=\sqrt{\frac{\alpha}{3}}n, \quad |C|=|D|=\frac{1-\frac{2r}{3}\alpha  }{r}n- \sqrt{\frac{\alpha}{3}}n, \quad |X_i|= \frac{1-\frac{2r}{3}\alpha  }{r}n.$$

The smallest class is $X$ and the second smallest are $A$ and $B$. By deleting all edges between $X$ and $A$ ($|X||A|= \frac{2r}{3\sqrt{3}}  \alpha^{3/2}n^2$) we get an $r$-partite graph. Since the classes $A$ and $X$ are the two smallest class sizes, one cannot do better as observed in \cite[Theorem 7]{Howmany}. Hence $$D_r(G)\geq \frac{2r}{3\sqrt{3}}  \alpha^{3/2}n^2.$$
Let us now count the number of edges of $G$. The number of edges incident to $X$ is 
\begin{align} e(X,X^c)&= \left(\frac{2r}{3}\alpha \right)\left(2\sqrt{\frac{\alpha}{3}} \right)n^2+ \left(\frac{2r}{3}\alpha \right)\left(\frac{1-\frac{2r}{3}\alpha  }{r} (r-2)\right)n^2 \\
&= \left(\frac{2}{3}(r-2)\alpha + \frac{4r}{3\sqrt{3}} \alpha^{3/2} - \frac{4r(r-2)}{9}\alpha^2    \right)n^2.   
\end{align}
Using that $|A|+|C|=|B|+|D|=|X_1|$, we have that the number of edges inside $A \cup B \cup C \cup D \cup X_1 \cup \dots \cup X_{r-2}$ is 
\begin{align} e(X^c)&= |X_1|^2 \binom{r}{2}- |A||B|= \left( \frac{1-\frac{2r}{3}\alpha  }{r}n \right)^2 \binom{r}{2} - \frac{1}{3}\alpha n^2  \\
&= \frac{1}{r^2}\binom{r}{2}n^2 -\frac{4r}{3} \frac{1}{r^2}\alpha \binom{r}{2}n^2 + \frac{4}{9}\alpha^2\binom{r}{2}n^2- \frac{1}{3}\alpha n^2  \\
&= \left(1-\frac{1}{r} \right)\frac{n^2}{2} - \frac{2}{3} (r-1) \alpha n^2 - \frac{1}{3}\alpha n^2 + \frac{4}{9}\alpha^2\binom{r}{2}n^2.
\end{align}
Thus, the number of edges of $G$ is
%\begin{align}
%    e(G)&= e(X)+e(X,X^c)=  \left(1-\frac{1}{r} %\right)\frac{n^2}{2}  + \left( \frac{2}{3}(r-2)- \frac{2}{3} %(r-1)  - \frac{1}{3} \right)   \alpha n^2 + \text{ LOT } n^2 % \\
%    &=  \left(1-\frac{1}{r} \right)\frac{n^2}{2}  + \left( %-\frac{4}{3}+ \frac{2}{3}   - \frac{1}{3} \right)   \alpha %n^2 + \text{ LOT } n^2  \\
%    &=  \left(1-\frac{1}{r} \right)\frac{n^2}{2}   -   \alpha %n^2 + \text{ LOT } n^2  
%    \end{align}
\begin{align}
    e(G) &= e(X^c)+e(X,X^c) = \left(1-\frac{1}{r} \right)\frac{n^2}{2}   -   \alpha n^2 + \frac{4r}{3\sqrt{3}} \alpha^{3/2}n^2 - \frac{2r(r-3)}{9} \alpha^2 n^2 \\
    &\geq \ex(n,K_{r+1})  -\alpha n^2 + \frac{4r}{3\sqrt{3}} \alpha^{3/2}n^2  -\frac{2r(r-3)}{9}\alpha^2 n^2,
        \end{align}
where we applied Tur\'an's theorem in the last step.          
\end{proof}

\bibliographystyle{abbrv}
\bibliography{rpartite}
\end{document}